\documentclass{birkjour}
 \newtheorem{thm}{Theorem}[section]
 \newtheorem{cor}[thm]{Corollary}

 \theoremstyle{definition}
 \newtheorem{defn}[thm]{Definition}
 \theoremstyle{remark}
 \newtheorem{rem}[thm]{Remark}
 \newtheorem*{ex}{Example}

 \numberwithin{equation}{section}

\begin{document}

%
%
%
%
%
%
%
%
%

\title[Decompositions generating codimension one manifold factors]
 {Decompositions of $\mathbb{R}^n, n \geq 4,$ into convex sets 
  generate codimension 1 manifold factors}


\author[D. M. Halverson]{Denise M. Halverson}

\address{Department of Mathematics\\
Brigham Young University\\
Provo, UT 84602}

\email{halverson@math.byu.edu}



\author[D. Repov\v s]{Du\v san Repov\v s}
\address{Faculty of Education and\\
Faculty of Mathematics and Physics\\
University of Ljubljana\\
P.O. Box 2964\\
Ljubljana, Slovenia 1001}

\email{dusan.repovs@guest.arnes.si}








\begin{abstract}
We show that if $G$ is an upper semicontinuous decomposition of
$\mathbb{R}^n$, $n \geq 4$, into convex sets, then the quotient
space $\mathbb{R}^n/G$ is a codimension one manifold factor.  In
particular, we show that $\mathbb{R}^n/G$ has the disjoint arc-disk
property.
\end{abstract}

\maketitle

\section{Introduction}

A space $X$ is said to be a \emph{codimension one manifold factor}
provided that $X \times \mathbb{R}$ is a manifold.  It is a long
standing unsolved problem as to whether all resolvable generalized
manifolds are codimension one manifold factors \cite{Daverman S}.
This is the so-called Product With a Line Problem and it is the
essence of the famous Generalized R.~L.~Moore Problem
\cite{Bertinoro, Moore1, Moore2}.

The Product With a Line Problem
speaks directly to one of the most fundamental questions in
geometric topology, which is how to recognize manifolds
\cite{Cannon, Edwards 2, Repovs1, Repovs2, Repovs3}. Because
manifolds have a rich structure which is useful to exploit in many
areas of mathematics and its applications, it is important to
recognize when one is dealing with a space that is a manifold.  One
notable example is the relevance of the Product With a Line Problem
to the famous Busemann Conjecture in metric geometry
\cite{Busemann1,Busemann2,Busemann3, HaRe}.

One might wonder even if a decomposition of $\mathbb{R}^n$ into
convex sets could  give rise to a decomposition space topologically
distinct from $\mathbb{R}^n$. This problem was investigated for
several years beginning with  Bing in the 1950's
\cite{Armentrout,Bing,CannonLO,Eaton, McAuley}. In 1970, Armentrout \cite{Armentrout} 
produced the first example of a decomposition of $\mathbb{R}^3$ into convex sets that yields a non-manifold.
Then in 1975, Eaton \cite{Eaton}
 demonstrated
 that a certain decomposition of
$\mathbb{R}^3$ into points and straight line segments, originally
proposed by Bing \cite{Bing}, is
indeed topologically distinct from
$\mathbb{R}^3$.  Hence, this type of complexity is
significant.  It should also be noted that there are no known 
examples of a non-manifold resulting from a decomposition of $\mathbb{R}^{n\geq4}$ into convex sets.

In this paper we shall investigate how the type of complexity
represented by decompositions of $\mathbb{R}^n$ into convex sets can
affect the classification of a decomposition space as a codimension
one manifold factor. We shall demonstrate that decompositions of
$\mathbb{R}^n$, $n \geq 4$, into convex sets are always codimension
one manifold factors. In particular, we shall show that such spaces
have a particularly strong general position property, the disjoint
arc-disk property.

\section{Preliminaries}

We briefly review some basic definitions and notations.  Recall that
a map $f: X \to Y$ is said to be \emph{proper} if whenever $C$ is a compact
subset of $Y$, then $f^{-1}(C)$ is compact.

There are various equivalent definitions of upper semicontinuous
decompositions \cite{Daverman book}, but the following will be the
most useful for our purposes:

\begin{defn}
A decomposition $G$ of $M$ into compact sets is said to be
\emph{upper semicontinuous}
(usc)
if and only if the associated
decomposition map $\pi: M \to M/G$ is a proper map.
\end{defn}

A compact subset $C$ of a space $X$ is said to be \emph{cell-like}
if for each neighborhood $U$ of $C$ in $X$, $C$ can be contracted to
a point
inside
$U$ \cite{MiRe}. A usc decomposition $G$ of $M$ is said to be
 \emph{cell-like} if each element $g \in G$ is
cell-like. A map $f: Y \to X$ is said to be
\emph{cell-like}
if
for each $x \in X$, $f^{-1}(x)$ is cell-like. A \emph{resolvable
generalized n-manifold} is an $n$-dimensional space $X$ that is the
image of a cell-like map $f:M \to X$ where $M$ is an $n$-manifold.

Convex sets are contractible, and hence they are cell-like. Thus, a
usc decomposition $G$ of $\mathbb{R}^n$ into convex sets is a
cell-like decomposition and the associated decomposition map $\pi:
\mathbb{R}^n \to \mathbb{R}^n/G$ is a cell-like map.  The fact that
$\mathbb{R}^n/G$ is finite-dimensional follows from a result of
Zemke (see \cite[Theorem 5.2]{Zemke}). Therefore, in this setting,
$\mathbb{R}^n/G$ is a resolvable generalized $n$-manifold.

For resolvable generalized manifolds, we have the following very
useful approximate lifting theorem, which follows from \cite[Theorem
17.1 and Corollary 16.12B]{Daverman book}:

\begin{thm} \label{ALT} Suppose that $G$ is a cell-like
decomposition of a manifold $M$,
with decomposition map $\pi: M \to
M/G $, and that
the quotient space
$M/G$ is finite-dimensional.
Then for any map $f: Z \to
M/G$  of a finite-dimensional compact polyhedron $Z$,
and any $\epsilon >0$,
there exists
a map $F:Z \to M$ such that $\pi F$
is an $\epsilon$-approximation of $f$.
\end{thm}


General position properties are very useful in detecting codimension
one manifold factors \cite{Daverman-Halverson 2,Halverson
1,Halverson 2,Halverson 3, HaRe2}.  For our results, we shall only
need to employ the following:

\begin{defn}
A space $X$ is said to have the \emph{disjoint arc-disk property}
(DADP) provided that any two maps $\alpha: I \to X$ and $f: D^2 \to
X$ can be approximated by maps with disjoint images, where $I$
denotes the unit interval and $D^2$ denotes a disk.
\end{defn}

\noindent The following theorem was demonstrated in \cite[Proposition
2.10]{Daverman 1}:

\begin{thm} \label{DADP thm} A resolvable generalized manifold having DADP is a codimension
one manifold factor.\end{thm}

Useful in discussions of the DADP is the local $0$-co-connectedness
property. A set $Z \subset X$ is said to have the \emph{local
$0$-co-connectedness property} ($0$-LCC) in $X$ if for every $z \in
Z \cap \text{Cl}(X-Z)$, each neighborhood $U$ of $z$ contains
another neighborhood $V$ of $z$ so that any two points in 
$V-Z$
are
path connected in 
$U-Z$.
Note that if $Z$ is nowhere dense in $X$,
then $Z =Z \cap \text{Cl}(X-Z)$.

The following theorem can be found in \cite[Corollary
26.2A]{Daverman book}:

\begin{thm} \label{0-LCC}
Each $k$-dimensional closed subset of a generalized $n$-manifold,
where $k \leq n-2$, is $0$-LCC.
\end{thm}

\noindent Since a $k$-dimensional closed subset of a generalized
$n$-manifold $X$, where $k \leq n-1$, is nowhere dense in $X$, we
have the following corollary: 

\begin{cor} \label{path}
If $Z$ is a $k$-dimensional closed subset of a generalized
$n$-manifold $X$, where $k \leq n-2$, then any path $\alpha: I \to
X$ can be approximated by a path  $\alpha': I \to X-Z$.
\end{cor}

\section{Main Results}

The main result of this paper is the following theorem:

\begin{thm}
Let $G$ be an upper semicontinuous decomposition of $\mathbb{R}^n$
into convex sets, where $n \geq 4$.  Then $\mathbb{R}^n/G$ is a
codimension one manifold factor.
\end{thm}

\noindent This theorem will follow immediately as a corollary of
Theorem \ref{DADP thm} and the following theorem:

\begin{thm} \label{DADP-4 thm}
Let $G$ be an upper semicontinuous decomposition of $\mathbb{R}^n$
into convex sets, where $n \geq 4$.  Then $\mathbb{R}^n/G$ has the
DADP.
\end{thm}

\begin{proof} Let $f: D^2 \to \mathbb{R}^n/G$ and $\varepsilon > 0$.
It follows from Corollary \ref{path} that it suffices to show that
there is an $\epsilon$-approximation $f': D^2 \to \mathbb{R}^n/G$ of
$f$ such that $f'(D^2)$ is $2$-dimensional.

Let $F: D^2 \to \mathbb{R}^n$ be a piecewise linear map, that is an
${\varepsilon}$-approximate lift of $f$. We shall show that $f' = \pi
F$ is then the
desired map.

Let $T$ denote a triangulation of $F(D^2)$. We claim that if
$\sigma$ is a $2$-simplex of $T$, then $f'(\sigma)$ is
$2$-dimensional. To see this, let $G_{\sigma}$ be the decomposition
induced over $\pi(\sigma)$, i.e. $G_{\sigma}$ is the decomposition
of $\mathbb{R}^n$ having as the only nontrivial elements, the
nontrivial elements of $G$ that meet $\sigma$.  Let $\omega:
\mathbb{R}^n \to \mathbb{R}^n/G_{\sigma}$ be the associated
decomposition map. Note that $\omega$ is necessarily a proper map,
being a decomposition induced over a closed set in the decomposition
space of a usc decomposition.

Let $P$ be the $2$-dimensional plane in  $ \mathbb{R}^n$ that contains $\sigma$. 
Let $\varpi$ denote the restriction of $\omega$ to $P$. Then $\varpi$ is also a
proper map.  Thus $\varpi$ determines a usc decomposition of the
plane into convex sets, elements that do not separate the plane. It
now follows from a classical result of Moore \cite{Moore1,Moore2},
that $\varpi$ is a near-homeomorphism onto its image. Thus
$\varpi(\sigma)$ is at most $2$-dimensional. 

But $\varpi(\sigma)$ is
homeomorphic to $\omega(\sigma)$, which in turn is homeomorphic to
$\pi(\sigma)$. Thus $\pi(\sigma)$ is at most $2$-di\-men\-sional
subset of $\mathbb{R}^n/G$. Hence $$f'(D^2) = \underset{\sigma \in
T^{(2)}}{\bigcup}\pi(\sigma)$$ is a $2$-dimensional subset of the
generalized $n$-manifold $\mathbb{R}^n/G$ \cite{HW}.
\end{proof}

\section{Conclusions}

As we have seen, the complexity represented by decompositions into
convex sets does not inhibit a decomposition space from being a
codimension one manifold factor. The fact that such spaces satisfy
the DADP is a pleasant result. 

It is well known that not all
codimension one manifold factors satisfy the DADP, and hence the
DADP is not a general position property that provides a
characterization of codimension one manifold factors. In fact, the
DADP condition is a
relatively weak tool for detecting codimension one
manifold factors, compared to other general position properties such
as:
\begin{itemize}
\item the disjoint homotopies property \cite{Halverson 1};
\item the plentiful $2$-manifolds property \cite{Halverson 1};
\item the $0$-stitched disks property \cite{Halverson 3};
\item the method of $\delta$-fractured maps \cite{Halverson 2}; and
\item the disjoint topographies (or disjoint concordance) property
\cite{Daverman-Halverson 2, HaRe2}. 
\end{itemize}
It is these stronger properties
that must be utilized to demonstrate that spaces such as the Totally
Wild Flow \cite{Cannon-Daverman2} and the Ghastly Spaces
\cite{Daverman-Walsh} are codimension one manifold factors.

In conclusion, we have demonstrated that we must look to other types
of complexities to realize a counterexample to the Generalized R.~L.~Moore
Problem, if  such an example does indeed exist.


\subsection*{Acknowledgment}
Supported by
the Slovenian Re\-search Agency grants
BI-US/11-12/023,
P1-0292-0101, J1-2057-0101 and
J1-4144-0101.
We thank the referee for several comments and suggestions.

\end{document}